\documentclass[11pt]{article}
\usepackage[a4paper,margin=2.65cm]{geometry}
\usepackage{amsmath,amssymb,amsthm,mathtools,mathrsfs}
\usepackage{booktabs,enumitem,microtype,tabularx}
\usepackage[hidelinks]{hyperref}
\numberwithin{equation}{section}

\newtheorem{theorem}{Theorem}[section]
\newtheorem{lemma}[theorem]{Lemma}
\newtheorem{proposition}[theorem]{Proposition}
\newtheorem{corollary}[theorem]{Corollary}
\theoremstyle{definition}
\newtheorem{definition}[theorem]{Definition}
\theoremstyle{remark}
\newtheorem{remark}[theorem]{Remark}

\DeclareMathOperator{\Hol}{Hol}
\newcommand{\D}{\mathbb D}
\newcommand{\T}{\mathbb T}
\newcommand{\Ccal}{\mathscr C}
\newcommand{\Pcal}{\mathcal P}
\newcommand{\Piop}{\boldsymbol\Pi}

\title{Vector-Carleson, Calder\'on, and Poisson--Atomic Criteria for\\
Generalized Hilbert Operators on $H^p$, $p>2$}
\author{Ma Yicen}
\date{}

\begin{document}
\maketitle

\begin{abstract}
Let
\[
 \mathcal H_gf(z)=\int_0^1f(t)g'(tz)\,dt
\]
and let $2<p<\infty$. Put $t=2p/(p-2)$ and
$X_j=2^{-j/p'}\Delta_jg'$, where $\Delta_j$ is the hard dyadic Taylor
projection. We prove that $\mathcal H_g:H^p\to H^p$ is bounded if and
only if
\[
 h\longmapsto(X_jh)_{j\ge0}:H^t\longrightarrow\ell^t(H^2)
\]
is bounded. The square of this embedding norm equals the norm of the
positive column operator $b\mapsto\sum_jb_j|X_j|^2$ from $\ell^{p/2}$
to $L^{p/2}$. Coordinate tails give essential-norm estimates and an exact
compactness criterion; Hardy duality gives an equivalent coanalytic
paraproduct inequality. The criterion is quantitatively invariant under
admissible smooth analytic dyadic resolutions and defines a
resolution-independent analytic Calder\'on space $\mathfrak C_p$ equal to
the Hilbert-range multiplier space. Its little space is simultaneously
the compact-symbol class, the polynomial closure, and the radial-norm
continuity class. The criterion yields a bounded noncompact symbol
with every large dyadic block densely occupied, outside both the lacunary
sector and the previously known blockwise sufficient space.

The full Hardy test ball can, without loss of norm, be replaced by outer
functions whose $t$th powers are finite positive mixtures of Poisson
kernels, provided the constant is uniform over all finite atom counts.
No fixed atom count suffices: even uniformly vanishing normalized-kernel
tests and a vanishing $H^p\to H^2$ mixed-norm tail may coexist with an
unbounded operator. For composite Dirichlet blocks we determine the sharp
atom count needed to capture their multiplier norm. Finally, finite dual
block weights produce intrinsic aggregate probability densities. Their
$L^1$ Poisson-mixture approximation gives an exact atomic-complexity
formula and an explicit finite-bandwidth testing bound.
\end{abstract}



\section{Introduction and main results}

For $\phi\in\Hol(\D)$ put
\[
 T_\phi f(z)=\int_0^1f(t)\phi(tz)\,dt,
 \tag{1.1}
\]
so that $T_{g'}=\mathcal H_g$. Guo and Tang identified
$T_\phi f=\phi*\mathcal Hf$ and characterized boundedness by the abstract
coefficient-multiplier space $(\mathcal R_p,H^p)$. They determined that
space for $1<p\le2$, while for $p>2$ they obtained strict upper and lower
inclusions \cite{GuoTang2026}; for broader coefficient-multiplier context
see \cite{BlascoPavlovic2011}. Norrbo, Pel\'aez and Wu characterized
$H^p\to H^2$ boundedness and solved the $H^p\to H^p$ problem for
lacunary symbols \cite{NorrboPelaezWu2026}. See also
\cite{GalanopoulosEtAl2014}.

For the classical Hilbert-matrix operator and its composition-operator
realization see \cite{DiamantopoulosSiskakis2000}; generalized Hilbert
operators on weighted Bergman spaces were developed in
\cite{PelaezRattya2013}, and further Hardy-space context appears in
\cite{Blasco2022}.  Background on coefficient multipliers and mixed-norm
spaces is provided by \cite{JevticVukoticArsenovic2016,Pavlovic2014}.
Vector-valued Carleson embeddings form a much broader subject; for a
nearby Bergman-space setting see \cite{ConstantinGavruta2015}.  None of
these general frameworks is used as a substitute for the explicit dyadic
reduction proved below.

The precise relation between the direct results and the present paper is
summarized below.  The last row is a theorem statement, not a claim of
global priority.

\begin{center}
\small
\begin{tabularx}{\textwidth}{@{}p{0.20\textwidth}XX@{}}
\toprule
Source & General $p>2$ result & Additional scope \\
\midrule
Guo--Tang \cite{GuoTang2026}
& Abstract characterization by $(\mathcal R_p,H^p)$ and strict classical
  upper/lower bounds
& Complete monotone-coefficient sector \\
Norrbo--Pel\'aez--Wu \cite{NorrboPelaezWu2026}
& Exact $H^p\to H^2$ mixed norm
& Complete lacunary-symbol sector and compactness there \\
Present paper
& Exact $H^t\to\ell^t(H^2)$ vector criterion for every analytic symbol
& Resolution-independent Calder\'on space, essential norm, compactness,
  Poisson-atomic testing, and testing complexity \\
\bottomrule
\end{tabularx}
\end{center}

Let $m$ be normalized Haar measure on $\T$, and define
\[
 \Delta_j\phi(z)=\sum_{2^j\le n<2^{j+1}}\widehat\phi(n)z^n,
 \qquad X_j=2^{-j/p'}\Delta_j\phi .
 \tag{1.2}
\]
Throughout,
\[
 2<p<\infty,\quad r=\frac p2,\quad
 t=\frac{2p}{p-2}=2r',\quad
 q=t'=\frac{2p}{p+2},\quad
 \alpha=\frac2t=\frac{p-2}{p}.
 \tag{1.3}
\]

Our principal theorem is the following explicit criterion.

\begin{theorem}[Vector-Carleson characterization]\label{thm:main}
The following are equivalent:
\begin{enumerate}[label=\textup{(\roman*)}]
 \item $T_\phi:H^p\to H^p$ is bounded;
 \item
 \[
  \Ccal_{\phi,p}h=(X_jh)_{j\ge0}:H^t\longrightarrow\ell^t(H^2)
  \tag{1.4}
 \]
 is bounded;
 \item every $h\in H^t$ satisfies
 \[
  \left(\sum_{j\ge0}
  [2^{-j/p'}\|h\Delta_j\phi\|_2]^t\right)^{1/t}
  \lesssim\|h\|_t.
  \tag{1.5}
 \]
\end{enumerate}
Moreover,
\[
 \|T_\phi\|\asymp_p|\phi(0)|+\|\Ccal_{\phi,p}\|,
 \qquad
 \|T_\phi\|_{\rm e}\asymp_p
 \lim_{N\to\infty}\|\Ccal_{\phi,p}^{>N}\|,
 \tag{1.6}\label{eq:main-norm-tail}
\]
where $\Ccal_{\phi,p}^{>N}h=(X_jh)_{j>N}$. In particular,
\[
 T_\phi\text{ is compact}
 \Longleftrightarrow
 \|\Ccal_{\phi,p}^{>N}\|\longrightarrow0.
 \tag{1.7}\label{eq:main-compact}
\]
It suffices in \textup{(iii)} to test zero-free outer functions $h$ for
which $1/h\in H^\infty$.
\end{theorem}

The hard-block formula in Theorem~\ref{thm:main} is not tied to that
particular cutoff. Theorem~\ref{thm:resolution} proves quantitative
invariance under admissible smooth analytic resolutions, and
Theorem~\ref{thm:calderon} packages the common norm as an intrinsic
analytic Calder\'on space $\mathfrak C_p=\mathfrak M_p$, including its
little space, compactness, and quotient essential norm.

The constant test $h=1$ is the dyadic form of the $H^p\to H^2$
mixed-norm condition. Thus Theorem~\ref{thm:main} identifies exactly the
additional simultaneous angular testing needed for $H^p$ output. Its
Hardy dual is the paraproduct synthesis map
\[
 \Piop_{\phi,p}(u_j)=P_+\left(\sum_j\overline{X_j}u_j\right):
 \ell^q(H^2)\longrightarrow H^q.
 \tag{1.8}\label{eq:paraproduct-intro}
\]
Theorem~\ref{thm:dense-separation} below also gives a dense-frequency
bounded symbol which lies outside the known sufficient condition
$(\|X_j\|_p)_j\in\ell^t$.

There is also a finite, though unbounded-complexity, testing description.
For $a\in\D$ let
\[
 P_a(\zeta)=\frac{1-|a|^2}{|1-\overline a\zeta|^2}.
\]
Denote by $\mathcal A_K$ the outer $H^t$ functions satisfying
\[
 |h|^t=\sum_{\nu=1}^Kc_\nu P_{a_\nu},\qquad
 c_\nu\ge0,\qquad\sum_{\nu=1}^Kc_\nu=1.
 \tag{1.9}
\]

\begin{theorem}[Exact finite Poisson-atomic testing]\label{thm:atomic}
With equality in $[0,\infty]$,
\[
 \|\Ccal_{\phi,p}\|
 =\sup_{K\ge1}\sup_{h\in\mathcal A_K}
 \left(\sum_j\|X_jh\|_2^t\right)^{1/t}.
 \tag{1.10}\label{eq:atomic-exact}
\]
The same equality holds for every coordinate tail. Consequently,
\eqref{eq:atomic-exact} characterizes boundedness, its tail characterizes compactness,
and the limiting tail is comparable to $\|T_\phi\|_{\rm e}$.
\end{theorem}

The supremum over all finite $K$ cannot be replaced by a fixed level.
For every $K_0<\infty$ there is a holomorphic $\phi$ for which all
$\mathcal A_{K_0}$ tests have a uniform vanishing tail, and the
$H^p\to H^2$ mixed-norm tail vanishes, yet $T_\phi$ is unbounded. The
case $K_0=1$ is a strong failure of the normalized reproducing-kernel
thesis.

This obstruction is quantitative. For
\[
 D_M(z)=1+\cdots+z^{M-1},\qquad
 X(z)=\frac d{\sqrt M}z^LD_M(z^N),
 \tag{1.11}
\]
we prove
\[
 \sup_{h\in\mathcal A_K}\|Xh\|_2
 \asymp_p d\left[1+M^{1/t}
 \min\{1,(K/N)^{1/p}\}\right].
 \tag{1.12}
\]
Hence a fixed proportion of the full multiplier norm
$\|M_X:H^t\to H^2\|\asymp_pdM^{1/t}$ is captured, as $M\to\infty$,
exactly at the scale $K\asymp N$.

Finally, for a finite nonnegative $\ell^r$-unit vector $\lambda$, put
\[
 S_\lambda=\sum_j\lambda_j|X_j|^2,\qquad
 A_\lambda=\|S_\lambda\|_r,\qquad
 W_\lambda=\frac{S_\lambda^r}{A_\lambda^r}.
 \tag{1.13}\label{eq:aggregate-def-intro}
\]
The $W_\lambda$ are probability densities. The exact $K$-atomic norm is
the supremum of $A_\lambda$ times the maximal affinity of $W_\lambda$
with a $K$-Poisson mixture; see Theorem~\ref{thm:aggregate}. Thus the
testing complexity is precisely Poisson-mixture approximation complexity
of near-norming aggregate densities. If a finite block family has angular
bandwidth at most $D$, this gives
\[
 1-\frac{\Ccal_K^2}{\Ccal_\infty^2}
 \lesssim_p
 \min\left\{1,
 \left[\sqrt{\frac{D+1}{K}}
 \log(eK(D+1))\right]^\alpha\right\}.
 \tag{1.14}\label{eq:bandwidth-intro}
\]

The exact formulas below were not found in the two direct 2026 papers or
in targeted searches; the global priority search is not exhaustive. We
therefore establish the explicit vector-Carleson and Poisson-atomic
characterizations below, while making no claim of global priority, of
priority for the classical ingredients, or of an identification with an
already named scalar smoothness space.

\section{Hardy and dyadic preliminaries}

We use standard Hardy-space facts from \cite{Duren1970,Garnett2007}.  The
coefficient multiplier needed in the proof is recorded explicitly.

\begin{lemma}[Dyadic Marcinkiewicz multiplier]\label{lem:dyadic-marc}
Let $1<u<\infty$ and let $\lambda=(\lambda_n)_{n\ge0}$ be a complex
sequence such that
\[
 \mathfrak M(\lambda)
 :=\sup_{n\ge0}|\lambda_n|
 +\sup_{j\ge0}
 \sum_{k=2^j}^{2^{j+1}-1}|\lambda_{k+1}-\lambda_k|<\infty.
 \tag{2.1}\label{eq:marc-condition}
\]
Then, for every $f(z)=\sum_{n\ge0}a_nz^n\in H^u$,
\[
 M_\lambda f(z):=\sum_{n\ge0}\lambda_na_nz^n\in H^u,
 \qquad
 \|M_\lambda f\|_{H^u}\le C_u\mathfrak M(\lambda)\|f\|_{H^u}.
 \tag{2.2}\label{eq:marc-conclusion}
\]
\end{lemma}

The sum in \eqref{eq:marc-condition} includes the jump from the last
integer in the $j$th block to the first integer in the next block.  This is
exactly \cite[Lemma~2.2]{GuoTang2026}, which in turn cites
\cite[Vol.~II, p.~232, Theorem~4.14]{Zygmund2002}.

We also use the analytic Littlewood--Paley equivalence
\[
 \left\|\sum_jF_j\right\|_u
 \asymp_u\left\|\left(\sum_j|F_j|^2\right)^{1/2}\right\|_u
 \tag{2.3}
\]
for functions supported in successive hard dyadic blocks; see
\cite{Zygmund2002,Grafakos2014}.

\begin{lemma}[Radial moment envelope]\label{lem:moment}
For $1<u<\infty$ and $f\in H^u$, put
\[
 A_j(f)=\int_0^1t^{2^j}|f(t)|\,dt,
 \qquad w_j(f)=2^{j/u'}A_j(f).
\]
Then $\|(w_j(f))\|_{\ell^u}\lesssim_u\|f\|_{H^u}$.
\end{lemma}

\begin{proof}
Partition $[0,1)$ into $J_0=[0,1/2)$ and
$J_k=[1-2^{-k},1-2^{-(k+1)})$, $k\ge1$, and put
$v_k=(\int_{J_k}|f|^u)^{1/u}$. The Fej\'er--Riesz radial restriction
inequality gives $\|v\|_{\ell^u}\lesssim_u\|f\|_{H^u}$. H\"older's
inequality on $J_k$ yields
\[
 w_j(f)\lesssim_u
 \sum_{k\le j}2^{(j-k)/u'}e^{-c2^{j-k}}v_k
 +\sum_{k>j}2^{-(k-j)/u'}v_k.
\]
Both displayed kernels are in $\ell^1$, and Young's inequality finishes
the proof.
\end{proof}

\begin{lemma}[Positive Hardy packets]\label{lem:packets}
For $N_j=2^j$, $\rho_j=N_j/(N_j+1)$, let
\[
 F_j(z)=\left(\frac{1-\rho_j^2}{(1-\rho_jz)^2}\right)^{1/u}.
 \tag{2.4}\label{eq:packet-def}
\]
Then $F_j(t)>0$ for $0\le t<1$, $\|F_j\|_{H^u}=1$, and
\[
 \left\|\sum_jv_jF_j\right\|_{H^u}\lesssim_u\|v\|_{\ell^u}.
 \tag{2.5}\label{eq:packet-synthesis}
\]
Moreover, if $2^j\le n<2^{j+1}$, then
\[
 \int_0^1t^nF_j(t)\,dt\gtrsim_u2^{-j/u'}.
 \tag{2.6}\label{eq:packet-moment}
\]
\end{lemma}

\begin{proof}
Since $\operatorname{Re}(1-\rho_jz)>0$ on $\D$, the analytic branch in
\eqref{eq:packet-def} is unambiguous.  Its boundary modulus satisfies
\[
 |F_j(e^{i\theta})|^u=P_{\rho_j}(e^{i\theta}),
\]
so $\|F_j\|_{H^u}=1$, and the chosen branch is positive on $[0,1)$.

Put $\delta_j=1-\rho_j=(2^j+1)^{-1}\asymp2^{-j}$.  The elementary
estimate
\[
 |1-\rho_je^{i\theta}|\asymp\delta_j+|\theta|
 \qquad(-\pi\le\theta\le\pi)
\]
gives
\[
 |F_j(e^{i\theta})|
 \lesssim_u\frac{\delta_j^{1/u}}
 {(\delta_j+|\theta|)^{2/u}}.
 \tag{2.7}\label{eq:packet-pointwise}
\]
Let $E_0=\{1/2<|\theta|\le\pi\}$ and, for $k\ge1$, let
$E_k=\{2^{-k-1}<|\theta|\le2^{-k}\}$.  These arcs cover $\T$ up to a
null set and $m(E_k)\asymp2^{-k}$.  On $E_k$, \eqref{eq:packet-pointwise}
implies
\[
 |F_j(e^{i\theta})|\lesssim_u
 \begin{cases}
  2^{j/u},&j\le k,\\
  2^{(2k-j)/u},&j>k.
 \end{cases}
\]
Therefore, pointwise on $E_k$,
\[
 2^{-k/u}\sum_j|v_j||F_j(e^{i\theta})|
 \lesssim_u\sum_j2^{-|j-k|/u}|v_j|.
 \tag{2.8}\label{eq:packet-convolution}
\]
After taking the $u$th power, integrating over $E_k$, and summing in
$k$, we obtain
\[
 \left\|\sum_jv_jF_j\right\|_{H^u}^u
 \lesssim_u
 \sum_{k\ge0}\left(\sum_j2^{-|j-k|/u}|v_j|\right)^u.
\]
The sequence $(2^{-|m|/u})_{m\in\mathbb Z}$ belongs to $\ell^1$, so
Young's convolution inequality proves \eqref{eq:packet-synthesis}.

For the moment estimate, set $N=2^j$ and suppose $N\le n<2N$.  On
$t\in[\rho_j,1]$,
\[
 t^n\ge\rho_j^{2N}\ge c,
 \qquad
 1-\rho_jt\le1-\rho_j^2\le\frac2{N+1},
 \qquad
 1-\rho_j^2\ge\frac1{N+1}.
\]
It follows that $F_j(t)\ge c_u(N+1)^{1/u}$ throughout an interval of
length $(N+1)^{-1}$.  Hence
\[
 \int_0^1t^nF_j(t)\,dt
 \ge c_u(N+1)^{-1/u'}\asymp_u2^{-j/u'},
\]
which is \eqref{eq:packet-moment}.
\end{proof}

\section{The dyadic reduction}

\begin{theorem}[Dyadic synthesis]\label{thm:dyadic}
For $1<u<\infty$, define on finite sequences
\[
 \mathscr S_{\phi,u}a
 =a_{-1}\phi(0)+\sum_{j\ge0}a_j2^{-j/u'}\Delta_j\phi.
 \tag{3.1}
\]
Then $T_\phi:H^u\to H^u$ is bounded if and only if
$\mathscr S_{\phi,u}:\ell^u(\{-1,0,1,\ldots\})\to H^u$ is bounded,
and the norms are comparable with constants depending only on $u$.
\end{theorem}

\begin{proof}
Assume $\mathscr S_{\phi,u}$ is bounded. For $f\in H^u$ put
$m_n=\int_0^1t^nf(t)dt$ and
$A_j=\int_0^1t^{2^j}|f(t)|dt$. Lemma~\ref{lem:moment} controls
$(2^{j/u'}A_j)$ in $\ell^u$. Hence the series
\[
 U:=\sum_{j\ge0}A_j\Delta_j\phi
   =\mathscr S_{\phi,u}\bigl(0,(2^{j/u'}A_j)_{j\ge0}\bigr)
 \tag{3.2}\label{eq:dyadic-U}
\]
converges in $H^u$ and satisfies
$\|U\|_{H^u}\lesssim_u\|\mathscr S_{\phi,u}\|\|f\|_{H^u}$.
Set $d_0=0$. On the $j$th block set
$d_n=m_n/A_j$ if $A_j>0$, and $d_n=0$ otherwise. Then $|d_n|\le1$ and
\[
 \sum_{n=2^j}^{2^{j+1}-2}|d_{n+1}-d_n|
 \le A_j^{-1}\int_0^1|f(t)|
 \sum_{n=2^j}^{2^{j+1}-2}t^n(1-t)dt\le1.
 \tag{3.3}\label{eq:dyadic-variation}
\]
For every $j$, the remaining jump
$|d_{2^{j+1}}-d_{2^{j+1}-1}|$ is at most $2$. Thus
$\mathfrak M(d)\le4$ in the precise convention of
Lemma~\ref{lem:dyadic-marc}.  The coefficient of $z^n$, $n\ge1$, in
$M_dU$ is $\widehat\phi(n)m_n$, so $M_dU$ is exactly the
positive-frequency part of $T_\phi f$.  Lemma~\ref{lem:dyadic-marc} and
\eqref{eq:dyadic-U} therefore give
\[
 \|T_\phi f-\widehat\phi(0)m_0\|_{H^u}
 \lesssim_u\|\mathscr S_{\phi,u}\|\|f\|_{H^u}.
\]
The Fej\'er--Riesz radial restriction inequality gives
$|m_0|\le\int_0^1|f(t)|dt\lesssim_u\|f\|_{H^u}$, while testing
$\mathscr S_{\phi,u}$ on the $-1$ coordinate gives
$|\widehat\phi(0)|\le\|\mathscr S_{\phi,u}\|$. This proves the required
upper estimate.

Conversely, suppose $T_\phi$ is bounded and first let $a_j\ge0$ be
finite. Put $f_a=\sum_ja_jF_j$. Lemma~\ref{lem:packets} controls its
$H^u$ norm. Its moments $M_n=\int_0^1t^nf_a(t)dt$ are positive and
decreasing, and for $2^j\le n<2^{j+1}$,
\[
 M_n\gtrsim_u a_j2^{-j/u'}.
 \tag{3.4}\label{eq:packet-moment-lower}
\]
Set $\eta_0=0$ and, for $2^j\le n<2^{j+1}$, set
\[
 \eta_n=\begin{cases}
 a_j2^{-j/u'}/M_n,&a_j>0,\\
 0,&a_j=0.
 \end{cases}
\]
By \eqref{eq:packet-moment-lower}, $\sup_n|\eta_n|\le C_u$. Since
$M_n$ is decreasing, $\eta_n$ is increasing on each dyadic block; hence
its total variation inside that block is at most $C_u$.  The jump to the
next block is at most $2C_u$. Consequently
$\mathfrak M(\eta)\le4C_u$, including the cross-block jump required in
\eqref{eq:marc-condition}.  Now the coefficient of $z^n$ in
$T_\phi f_a$ is $\widehat\phi(n)M_n$, so Lemma~\ref{lem:dyadic-marc}
gives
\[
 \left\|\sum_{j\ge0}a_j2^{-j/u'}\Delta_j\phi\right\|_{H^u}
 \le C_u\|T_\phi f_a\|_{H^u}
 \lesssim_u\|T_\phi\|\,\|a\|_{\ell^u}.
 \tag{3.5}\label{eq:positive-synthesis}
\]
Every complex sequence is the sum of the positive and negative parts of
its real and imaginary parts, each having $\ell^u$ norm no larger than
the original norm. Applying \eqref{eq:positive-synthesis} four times
proves the nonconstant synthesis estimate. Finally, $T_\phi1(0)=\phi(0)$
controls the $-1$ coordinate and completes the proof.
\end{proof}

For $u=p$, Littlewood--Paley theory converts this theorem to a positive
operator. Define
\[
 \Pcal_{\phi,p}:\ell^r\to L^r,
 \qquad \Pcal_{\phi,p}b=\sum_{j\ge0}b_j|X_j|^2.
 \tag{3.6}\label{eq:positive-column}
\]
For finite $a$,
\[
 \left\|\sum_ja_jX_j\right\|_p^2
 \asymp_p\left\|\sum_j|a_j|^2|X_j|^2\right\|_r,
 \qquad\|(|a_j|^2)\|_{\ell^r}=\|a\|_{\ell^p}^2.
\]
Since \eqref{eq:positive-column} is positive,
\[
 \|T_\phi\|\asymp_p|\phi(0)|+\|\Pcal_{\phi,p}\|^{1/2}.
 \tag{3.7}\label{eq:positive-norm}
\]

\subsection{Resolution invariance}

\begin{definition}[Admissible analytic resolution]\label{def:admissible}
Let
\[
 I_k=\{n\in\mathbb N:2^k\le n<2^{k+1}\},\qquad k\ge0.
\]
A family of scalar sequences
$\sigma=(\sigma_j)_{j\ge0}$ on $\mathbb N$ is called
$(L,M,p)$-admissible if:
\begin{enumerate}[label=\textup{(\alph*)}]
 \item $\sigma_j(n)=0$ for $n\in I_k$ whenever $|j-k|>L$;
 \item $\sum_{j\ge0}\sigma_j(n)=1$ for every $n\ge1$;
 \item for every $|\ell|\le L$, the patched sequence
 \[
  m_\ell(n)=\sigma_{k+\ell}(n)\qquad(n\in I_k),
  \tag{3.8}\label{eq:patched-multiplier}
 \]
 with $\sigma_j=0$ for $j<0$, is an $H^p$ coefficient multiplier of
 norm at most $M$.
\end{enumerate}
Put
\[
 W_j^\sigma\phi(z)=\sum_{n\ge1}\sigma_j(n)\widehat\phi(n)z^n,
 \qquad
 \mathscr S_{\phi,p}^{\sigma}a
 =a_{-1}\phi(0)+
  \sum_{j\ge0}a_j2^{-j/p'}W_j^\sigma\phi.
 \tag{3.9}\label{eq:smooth-synthesis}
\]
The hard resolution is $\sigma_j=\mathbf1_{I_j}$ and has $L=0$.
\end{definition}

The usual compactly supported smooth analytic Littlewood--Paley
partitions are admissible: finite overlap and the partition property are
immediate, while \eqref{eq:patched-multiplier} follows from
Lemma~\ref{lem:dyadic-marc}; see also
\cite{Zygmund2002,Grafakos2014}.

\begin{theorem}[Resolution invariance]\label{thm:resolution}
Let $2<p<\infty$ and let $\sigma$ be an $(L,M,p)$-admissible analytic
resolution. Then
\[
 \boxed{
 \|\mathscr S_{\phi,p}^{\sigma}:\ell^p\to H^p\|
 \asymp_{p,L,M}
 \|\mathscr S_{\phi,p}:\ell^p\to H^p\|.}
 \tag{3.10}\label{eq:resolution-norm}
\]
For the nonconstant coordinate tails,
\[
\begin{split}
 \|\mathscr S_{\phi,p}^{\sigma,>N}\|
 &\lesssim_{p,L,M}
 \|\mathscr S_{\phi,p}^{>N-L}\|,\\
 \|\mathscr S_{\phi,p}^{>N}\|
 &\lesssim_{p,L}
 \|\mathscr S_{\phi,p}^{\sigma,>N-L}\|,
\end{split}
 \tag{3.11}\label{eq:resolution-tails}
\]
where a negative cutoff means the full nonconstant synthesis.
Consequently,
\[
 \|T_\phi\|_{\mathrm e}
 \asymp_{p,L,M}
 \lim_{N\to\infty}\|\mathscr S_{\phi,p}^{\sigma,>N}\|,
 \qquad
 T_\phi\text{ is compact}
 \iff
 \|\mathscr S_{\phi,p}^{\sigma,>N}\|\longrightarrow0.
 \tag{3.12}\label{eq:resolution-essential}
\]
\end{theorem}

\begin{proof}
Write $w_j=2^{-j/p'}$ and extend all negative coordinates by zero.
Finite scale overlap gives, first for finite inputs,
\[
 \sum_ja_jw_jW_j^\sigma\phi
 =\sum_{\ell=-L}^{L}2^{-\ell/p'}M_{m_\ell}
   \left(\sum_{k\ge0}a_{k+\ell}w_k\Delta_k\phi\right).
 \tag{3.13}\label{eq:smooth-from-hard}
\]
This is an exact coefficient identity on each $I_k$.  The multiplier
assumption and the hard synthesis norm prove the first inequality in
\eqref{eq:resolution-norm}.

For the reverse inequality, the partition property gives
\[
 \sum_kb_kw_k\Delta_k\phi
 =\sum_{\ell=-L}^{L}\sum_{j\ge0}
 b_{j+\ell}w_{j+\ell}\Delta_{j+\ell}W_j^\sigma\phi.
 \tag{3.14}\label{eq:hard-from-smooth}
\]
Let $Q=2L+1$ and split $j$ into its residue classes modulo $Q$.
For fixed $\ell$ and one residue class $c$, set
\[
 a_j^{\ell,c}=
 \begin{cases}
 2^{-\ell/p'}b_{j+\ell},&j\equiv c\pmod Q,\\
 0,&\text{otherwise}.
 \end{cases}
\]
The hard-scale supports of $W_j^\sigma\phi$ belonging to one colour are
pairwise disjoint.  The multiplier whose symbol is the indicator of
$\bigcup_{j\equiv c\pmod Q}I_{j+\ell}$ therefore extracts precisely the
terms $\Delta_{j+\ell}W_j^\sigma\phi$.  Its symbol is constant on every
hard block, bounded by one, and has at most one jump in each variation
sum in \eqref{eq:marc-condition}; Lemma~\ref{lem:dyadic-marc} bounds it
uniformly on $H^p$. Since
\[
 a_j^{\ell,c}w_j=b_{j+\ell}w_{j+\ell},
 \qquad
 \|a^{\ell,c}\|_{\ell^p}
 \le2^{-\ell/p'}\|b\|_{\ell^p},
\]
summing over the finitely many $\ell$ and colours proves the reverse
inequality. Density extends the identities from finite inputs.

If $a$ is supported on $j>N$, every hard coordinate in
\eqref{eq:smooth-from-hard} satisfies $k>N-L$. Conversely, if $b$ is
supported on $k>N$, every smooth input in the coloured reconstruction
satisfies $j>N-L$. This proves \eqref{eq:resolution-tails}. Tail norms
are decreasing, and shifting a cutoff by the fixed integer $L$ does not
change their limits. The hard-tail essential-norm and compactness result
follows directly from Theorem~\ref{thm:dyadic}: the uniformly bounded hard
head projections have finite-dimensional range and converge strongly to
the identity, while their complementary tails kill compact operators in
operator norm.  Combining that hard-tail identity with
\eqref{eq:resolution-tails} proves \eqref{eq:resolution-essential}.
\end{proof}

\section{Vector embedding, tails, and the paraproduct dual}

\begin{theorem}\label{thm:exact}
$\Pcal_{\phi,p}:\ell^r\to L^r$ is bounded if and only if
$\Ccal_{\phi,p}:H^t\to\ell^t(H^2)$ is bounded, and
\[
 \boxed{\|\Ccal_{\phi,p}\|^2=\|\Pcal_{\phi,p}\|.}
 \tag{4.1}\label{eq:embedding-positive-exact}
\]
\end{theorem}

\begin{proof}
The positive adjoint is
\[
 (\Pcal_{\phi,p}^*F)_j=\int_\T|X_j|^2F\,dm,\qquad F\in L^{r'}_+.
\]
For $h\in H^t$, $1/2=1/p+1/t$ and $t=2r'$ give
\[
 \|\Pcal_{\phi,p}^*|h|^2\|_{\ell^{r'}}
 =\left(\sum_j\|X_jh\|_2^t\right)^{2/t},
 \qquad\||h|^2\|_{r'}=\|h\|_t^2.
 \tag{4.2}\label{eq:outer-positive-identity}
\]
Thus $\|\Ccal_{\phi,p}\|^2\le\|\Pcal_{\phi,p}\|$.

Conversely, for $F\in L^{r'}_+$ and $\varepsilon>0$,
$\log(F+\varepsilon)\in L^1$. Choose outer $h_\varepsilon$ with
$|h_\varepsilon|^2=F+\varepsilon$. Then $h_\varepsilon\in H^t$ and
$1/h_\varepsilon\in H^\infty$. Apply \eqref{eq:outer-positive-identity} and let
$\varepsilon\downarrow0$. Monotone convergence in the nonnegative
coordinate sum gives
\[
 \|\Pcal_{\phi,p}^*F\|_{\ell^{r'}}
 \le\|\Ccal_{\phi,p}\|^2\|F\|_{r'}.
\]
Positive duality proves
\eqref{eq:embedding-positive-exact}.
\end{proof}

This theorem and \eqref{eq:positive-norm} prove the boundedness part of
Theorem~\ref{thm:main}. For tails, let $Q_N$ retain the constant and the
blocks $j\le N$. These Fourier projections are uniformly bounded on
$H^p$, converge strongly to the identity, and
\[
 (I-Q_N)T_\phi=T_{\phi^{>N}},\qquad Q_NT_\phi\text{ has finite rank}.
\]
If $K$ is compact, $\|(I-Q_N)K\|\to0$. Therefore
\[
 \|T_\phi\|_{\rm e}\asymp_p\lim_N\|T_{\phi^{>N}}\|.
 \tag{4.3}
\]
Applying \eqref{eq:embedding-positive-exact} to $(X_j)_{j>N}$ proves
\eqref{eq:main-norm-tail}--\eqref{eq:main-compact}.
The condition concerns coordinate tails; it does not assert that
$\Ccal_{\phi,p}$ itself is compact.

\begin{corollary}[Paraproduct dual]\label{cor:dual}
$T_\phi:H^p\to H^p$ is bounded if and only if
\eqref{eq:paraproduct-intro} is bounded.
Moreover,
\[
 \|\Piop_{\phi,p}\|\asymp_p\|\Ccal_{\phi,p}\|,
 \quad
 \|T_\phi\|_{\rm e}\asymp_p\lim_N\|\Piop_{\phi,p}^{>N}\|,
\]
and the tail tends to zero exactly when $T_\phi$ is compact.
\end{corollary}

\begin{proof}
Because $1/q=1/p+1/2$ and $1<q<2$, $P_+$ is bounded on $L^q$. For
finite $u$ and $h\in H^t$,
\[
 \langle h,\Piop_{\phi,p}u\rangle
 =\sum_j\langle X_jh,u_j\rangle_{H^2}.
\]
Thus $\Piop_{\phi,p}$ is the analytic representative of
$\Ccal_{\phi,p}^*$. The standard identification $(H^t)^*\simeq H^q$
gives the comparison, uniformly for all tails.
\end{proof}

\subsection{A dense-frequency separation from the known sufficient space}

The next application shows that Theorem~\ref{thm:main} is not merely a
reformulation of the previously known blockwise sufficient condition.
Recall that $E\subset\mathbb N_0$ is an analytic $\Lambda(p)$-set if some
$C_E<\infty$ satisfies $\|f\|_{H^p}\le C_E\|f\|_{H^2}$ for every
analytic polynomial whose spectrum is contained in $E$.
We write $C_{E,p}$ for the least such constant.

\begin{theorem}[Dense-frequency separation]\label{thm:dense-separation}
For every $2<p<\infty$ there is $\phi\in\Hol(\D)$ such that:
\begin{enumerate}[label=\textup{(\roman*)}]
 \item $T_\phi:H^p\to H^p$ is bounded but not compact;
 \item every sufficiently large hard dyadic block is active and contains
 a consecutive frequency interval of length $2^{j-1}$;
 \item $\operatorname{supp}\widehat\phi$ is not an analytic
 $\Lambda(p)$-set;
 \item
 \[
  2^{-j/p'}\|\Delta_j\phi\|_{H^p}=1
  \qquad(j\ge j_0).
  \tag{4.4}
 \]
\end{enumerate}
In particular, the symbol lies outside the known sufficient class defined
by
\[
 \bigl(2^{-j/p'}\|\Delta_j\phi\|_{H^p}\bigr)_j\in\ell^t,
 \tag{4.5}\label{eq:known-block-sufficient}
\]
and it is neither lacunary nor spectrally thin.
\end{theorem}

\begin{proof}
Let
\[
 F_N(e^{i\theta})=\frac1{N+1}
 \left|\sum_{\nu=0}^Ne^{i\nu\theta}\right|^2.
\]
The standard Fej\'er-kernel estimates give
\[
 F_N(e^{i\theta})\lesssim
 \frac{N+1}{1+(N+1)^2\theta^2},
 \qquad
 \|F_N\|_r\asymp_r(N+1)^{1-1/r}.
 \tag{4.6}\label{eq:fejer-estimates}
\]
Set $c_N=\|F_N\|_r^{-1}$ and
\[
 H_N(z)=\sqrt{\frac{c_N}{N+1}}
 \sum_{\nu=0}^Nz^\nu,
 \qquad |H_N|^2=c_NF_N.
\]
Thus $\||H_N|^2\|_r=1$.  Moreover, for $1\le A\le N+1$,
integration of the first estimate in \eqref{eq:fejer-estimates} gives
\[
 \bigl\|c_NF_N\mathbf1_{\{|\theta|>A/(N+1)\}}\bigr\|_r
 \lesssim_r A^{-(2-1/r)}.
 \tag{4.7}\label{eq:fejer-tail}
\]

Choose $0<\eta<1$, put $N_j=2^{j-1}-1$ and $A_j=2^{\eta j}$.
The numbers $A_j/(N_j+1)$ form a summable sequence.  After increasing
$j_0$, choose pairwise disjoint arcs $I_j$ with those radii and centers
$e^{i\theta_j}$.  Define
\[
 K_j(z)=H_{N_j}(e^{-i\theta_j}z),
 \qquad Y_j=|K_j|^2,
 \qquad Z_j=Y_j\mathbf1_{I_j},
 \qquad E_j=Y_j-Z_j.
\]
The $Z_j$ have disjoint supports and $\|Z_j\|_r\to1$.  By
\eqref{eq:fejer-tail}, $(\|E_j\|_r)_j\in\ell^{r'}$.  Hence, for every finite $b$,
\[
 \left\|\sum_jb_jZ_j\right\|_r^r
 =\sum_j|b_j|^r\|Z_j\|_r^r
 \lesssim\|b\|_{\ell^r}^r,
\]
and
\[
 \left\|\sum_jb_jE_j\right\|_r
 \le\sum_j|b_j|\|E_j\|_r
 \le\|b\|_{\ell^r}
 \|(\|E_j\|_r)_j\|_{\ell^{r'}}.
 \tag{4.8}
\]
Thus $b\mapsto\sum_jb_jY_j$ is bounded from $\ell^r$ to $L^r$.
Its coordinate-tail norm does not tend to zero, because
$\|Y_j\|_r=1$ for every $j$.

Now set
\[
 \Delta_j\phi(z)=2^{j/p'}z^{2^j}K_j(z)
 \qquad(j\ge j_0),
\]
and set the earlier blocks to zero.  Since $\deg K_j=2^{j-1}-1$, every
displayed polynomial lies inside the $j$th hard block.  Its coefficients
grow at most polynomially in their actual frequency, so the block series
defines a holomorphic function on $\D$.  The normalized block is
$X_j=z^{2^j}K_j$, whence $|X_j|^2=Y_j$.  Theorems
\ref{thm:dyadic} and \ref{thm:exact}, together with the tail criterion,
show that $T_\phi$ is bounded but not compact.  Also
\[
 \|X_j\|_p^2=\||X_j|^2\|_r=\|Y_j\|_r=1,
\]
which proves \textup{(iv)} and the failure of
\eqref{eq:known-block-sufficient}.

Every coefficient of $K_j$ is nonzero, so the support contains an interval
of length $2^{j-1}$ in every sufficiently large block.  If this support
were an analytic $\Lambda(p)$-set, the shifted Dirichlet polynomials on
those intervals would have uniformly bounded $H^p/H^2$ ratio.  Instead,
for an interval of length $M$ that ratio is
$\asymp_pM^{1/2-1/p}\to\infty$.  This proves \textup{(ii)}--\textup{(iii)}.
\end{proof}

\subsection{Known regimes and an intrinsic Calder\'on symbol space}

\begin{corollary}[Recovery of known regimes and strict extension]
\label{cor:known-regimes}
Put
\[
 B_{2,p}(\phi)
 =|\phi(0)|+
 \left(\sum_{j\ge0}
  \bigl(2^{-j/p'}\|\Delta_j\phi\|_2\bigr)^t\right)^{1/t}.
 \tag{4.9}\label{eq:h2-block-condition}
\]
Then:
\begin{enumerate}[label=\textup{(\roman*)}]
 \item $\|T_\phi:H^p\to H^2\|\asymp_p B_{2,p}(\phi)$, and
 boundedness of $T_\phi:H^p\to H^p$ implies
 $B_{2,p}(\phi)<\infty$.
 \item If $\operatorname{supp}\widehat\phi$ is contained in an analytic
 $\Lambda(p)$-set $E\subset\mathbb N_0$, then
 \[
  T_\phi:H^p\to H^p\text{ is bounded}
  \iff B_{2,p}(\phi)<\infty,
 \]
 with
 \[
  \|T_\phi:H^p\to H^2\|
  \le\|T_\phi:H^p\to H^p\|
  \le C_{E,p}\|T_\phi:H^p\to H^2\|.
 \]
 In this sector boundedness is equivalent to compactness.
 \item Hadamard-lacunary symbols are a special case of \textup{(ii)}.
 For a one-frequency-per-block symbol the test $h=1$ exactly norms the
 vector criterion, up to constants depending only on $p$.
 \item The blockwise condition
 \[
  \bigl(2^{-j/p'}\|\Delta_j\phi\|_{H^p}\bigr)_j\in\ell^t
 \]
 is sufficient in general and is not necessary, even for symbols with
 a consecutive interval occupying half of every sufficiently large hard
 block.
\end{enumerate}
\end{corollary}

\begin{proof}
The first norm equivalence is the dyadic form of the
$H^p\to H^2$ theorem in \cite{NorrboPelaezWu2026}; necessity also follows
from Theorem~\ref{thm:main} by testing $h=1$.  If the output spectrum lies
in $E$, the defining $\Lambda(p)$ inequality extends from polynomials by
closure and gives
$\|T_\phi f\|_p\le C_{E,p}\|T_\phi f\|_2$.  The reverse norm inequality
uses the contractive inclusion $H^p\hookrightarrow H^2$.  The
$\ell^t$ tail in \eqref{eq:h2-block-condition} tends to zero, so finite
block truncation proves compactness.

Hadamard-gap sets are analytic $\Lambda(p)$-sets; alternatively, in the
one-frequency-per-block case multiplication by the corresponding monomial
is an isometry on $H^2$, so the vector norm is the constant-test norm.
Finally, H\"older gives
$\|X_jh\|_2\le\|X_j\|_p\|h\|_t$, proving the blockwise sufficient
condition.  Its strict failure is Theorem~\ref{thm:dense-separation}.
\end{proof}

\begin{lemma}[Analytic allocation factorization]\label{lem:allocation}
For a sequence $Y=(Y_j)_{j\ge0}\subset H^p$, set
\[
 P_Yb=\sum_jb_j|Y_j|^2
\]
and
\[
 \mathfrak G_p(Y)
 =\inf_{Y_j=A_jF_j}
 \|A\|_{H^\infty(\ell^t)}
 \sup_j\|F_j\|_{H^p},
 \tag{4.10}\label{eq:analytic-factor-norm}
\]
where the product is coordinatewise analytic. With equality in
$[0,\infty]$,
\[
 \boxed{\mathfrak G_p(Y)
 =\|P_Y:\ell^r\to L^r\|^{1/2}.}
 \tag{4.11}\label{eq:allocation-bounded}
\]
If $P_Y$ is bounded, then
\[
 \boxed{
 \inf_{Y=AF}\|A\|_{H^\infty(\ell^t)}
       \limsup_{j\to\infty}\|F_j\|_{H^p}
 =\|P_Y\|_{\mathrm e}^{1/2},}
 \tag{4.12}\label{eq:allocation-essential}
\]
and $P_Y$ is compact if and only if one such factorization has
$\|F_j\|_{H^p}\to0$.  In all these infima one may require the $A_j$ to
be outer and zero-free and
\[
 \sum_j|A_j^*|^t=1\quad\text{a.e. on }\T.
 \tag{4.13}\label{eq:outer-allocation}
\]
No attainment of an infimum is asserted.
\end{lemma}

\begin{proof}
Put $w_j=|Y_j^*|^2\in L^r$.  We first prove the measurable
change-of-density formula underlying the analytic statement.  Call
$d=(d_j)$ an allocation if $d_j\ge0$ and
$\sum_jd_j^{r'}=1$ a.e., and use the convention $0/0=0$.  We claim that
\[
 \inf_d\sup_j\|w_j/d_j\|_{L^r}=\|P_Y:\ell^r\to L^r\|.
\]
Indeed, for a finitely supported $b\ge0$, pointwise H\"older in the
counting variable gives
\[
 P_Yb
 =\sum_jd_jb_j\frac{w_j}{d_j}
 \le
 \left(\sum_jb_j^r\left(\frac{w_j}{d_j}\right)^r\right)^{1/r}.
\]
After integration,
\[
 \|P_Yb\|_r^r
 \le\sum_jb_j^r\|w_j/d_j\|_r^r,
\]
and positivity handles arbitrary complex $b$.  Thus the displayed
infimum is at least $\|P_Y\|$.

For the converse, fix $\kappa>\|P_Y\|$.  Theorem~4 of
Howard--Schep \cite{HowardSchep1990}, applied with equal exponents $r$
and with $\lambda=\kappa^r$, supplies a strictly positive
$u=(u_j)\in\ell^r$ such that
\[
 P_Y^*\bigl((P_Yu)^{r-1}\bigr)_j
 \le \kappa^r u_j^{r-1}\qquad(j\ge0).
\]
On $E^c=\{P_Yu>0\}$ define
\[
 d_j=\left(\frac{u_jw_j}{P_Yu}\right)^{1/r'}.
\]
On $E$ all the $w_j$ vanish a.e.; there we take any fixed strictly
positive $\ell^{r'}$-unit vector for $d$.  Then $d$ is an allocation,
and the preceding Howard--Schep inequality yields
\[
 \|w_j/d_j\|_r^r
 =u_j^{1-r}\int_\T w_j(P_Yu)^{r-1}\,dm
 \le\kappa^r.
\]
Letting $\kappa\downarrow\|P_Y\|$ proves
the claimed measurable change-of-density formula.  This also identifies precisely the
Howard--Schep result used here, rather than appealing to an abstract
factorization without displaying the resulting density.

Since $p=2r$ and $t=2r'$, putting $a_j=d_j^{1/2}$ transforms
that formula into
\[
 \|P_Y\|^{1/2}
 =\inf_{\substack{a_j\ge0\\\sum_ja_j^t=1}}
   \sup_j\|Y_j^*/a_j\|_{L^p}.
 \tag{4.14}\label{eq:measurable-allocation}
\]
Indeed,
$\|Y_j^*/a_j\|_p^2=\|w_j/d_j\|_r$.  If $P_Y$ is unbounded, any finite
factorization would give a finite upper bound for $P_Y$ by the first
H\"older estimate, so both sides of \eqref{eq:allocation-bounded} are
infinite.

Any analytic factorization in \eqref{eq:analytic-factor-norm} gives a
measurable suballocation
$|A_j^*|/\|A\|_{H^\infty(\ell^t)}$; completing the missing $t$-power
mass only increases denominators. This proves
$\|P_Y\|^{1/2}\le\mathfrak G_p(Y)$.

Conversely, choose an allocation in \eqref{eq:measurable-allocation},
a fixed strictly positive $\ell^t$-unit vector $\lambda$, and
$\varepsilon>0$. Put
\[
 \widetilde a_j
 =\frac{(a_j^t+\varepsilon^t\lambda_j^t)^{1/t}}
 {(1+\varepsilon^t)^{1/t}}.
 \tag{4.15}\label{eq:allocation-regularization}
\]
Then $\sum_j\widetilde a_j^t=1$,
$\widetilde a_j\ge a_j/(1+\varepsilon^t)^{1/t}$, and, for each fixed
$j$, $\widetilde a_j$ is bounded above and bounded away from zero.
Let $A_j$ be the outer function with boundary modulus
$\widetilde a_j$.  It is zero-free, $1/A_j\in H^\infty$, and
$F_j=Y_j/A_j\in H^p$. Outer Jensen and Tonelli give
\[
 \sum_j|A_j(z)|^t
 \le\mathcal P_z\left(\sum_j\widetilde a_j^t\right)=1.
\]
The map $A=(A_j)$ is strongly analytic as an $\ell^t$-valued map.  To
see the only nonformal point, for $0<R<1$ outer Jensen gives
\[
 \sup_{|z|\le R}\sum_{j>J}|A_j(z)|^t
 \le C_R\int_\T\sum_{j>J}\widetilde a_j^t\,dm\longrightarrow0.
\]
Thus the finite-coordinate analytic maps converge locally uniformly to
$A$.  Since countably many scalar outer functions have their prescribed
radial limits simultaneously off one null set,
\eqref{eq:outer-allocation} holds.  Moreover,
\[
 \|F_j\|_{H^p}
 \le(1+\varepsilon^t)^{1/t}\|Y_j^*/a_j\|_{L^p}.
\]
Letting the allocation error and then $\varepsilon$ tend to zero proves
\eqref{eq:allocation-bounded}.

We next prove the assertions involving the essential norm.  Write
$E_N$ for the projection onto the first $N+1$ coordinates.  The
positive-column tail identity is
\[
 \|P_Y\|_{\mathrm e}
 =\lim_{N\to\infty}\|P_Y^{>N}\|.
 \tag{4.16}\label{eq:positive-tail-essential}
\]
The inequality ``$\le$'' follows by deleting the finite-rank head.
For the reverse, if $K:\ell^r\to L^r$ is compact, then
$K^*:L^{r'}\to\ell^{r'}$ is compact.  Coordinate tails converge
uniformly on compact subsets of $\ell^{r'}$, whence
\[
 \|K(I-E_N)\|=\|(I-E_N)K^*\|\longrightarrow0.
\]
It follows that
$\lim_N\|P_Y(I-E_N)\|\le\|P_Y-K\|$; taking the infimum over $K$
proves \eqref{eq:positive-tail-essential}.

The same allocation argument gives the exact measurable tail formula
\[
 \inf_d\limsup_{j\to\infty}\|w_j/d_j\|_r
 =\|P_Y\|_{\mathrm e}.
 \tag{4.16a}\label{eq:measurable-essential}
\]
For any allocation $d$, the first H\"older estimate applied to indices
$j>N$ gives
$\|P_Y^{>N}\|\le\sup_{j>N}\|w_j/d_j\|_r$, proving one inequality.
Conversely, fix $N$ and choose, by the measurable change-of-density
formula above, an
allocation $d^{\rm tail}$ on $j>N$ whose cost is within $\eta$ of
$\|P_Y^{>N}\|$.  Choose a strictly positive allocation
$h=(h_0,\ldots,h_N)$ on the finite head.  For $0<\delta<1$ put
\[
 d_j=\delta h_j\quad(j\le N),\qquad
 d_j=(1-\delta^{r'})^{1/r'}d_j^{\rm tail}\quad(j>N).
\]
This is a global allocation, and its limsup cost is at most
$(1-\delta^{r'})^{-1/r'}(\|P_Y^{>N}\|+\eta)$.  First let
$N\to\infty$ and then $\eta,\delta\downarrow0$ to obtain the reverse
inequality in \eqref{eq:measurable-essential}.

Passing from $d_j$ to $a_j=d_j^{1/2}$, and then using the same
regularized outer lift, turns \eqref{eq:measurable-essential} into
\eqref{eq:allocation-essential}.  The multiplicative loss
$(1+\varepsilon^t)^{1/t}$ is uniform in $j$, so the lift preserves the
limsup after $\varepsilon\downarrow0$.

It remains to prove the attained vanishing assertion in the compact
case.  If $P_Y$ is compact, choose
$N_1<N_2<\cdots$ so that
$\|P_Y^{>N_m}\|\le2^{-2m}$ for $m\ge1$.  Set
\[
 I_0=\{0,\ldots,N_1\},\qquad
 I_m=\{N_m+1,\ldots,N_{m+1}\}\quad(m\ge1),
\]
and define a new positive column by $B_j=w_j$ on $I_0$ and
$B_j=2^mw_j$ on $I_m$.  If
$x_m=\|b\mathbf1_{I_m}\|_{\ell^r}$, then
\[
 \left\|\sum_jb_jB_j\right\|_r
 \le\|P_Y\|x_0+\sum_{m\ge1}2^m
       \|P_Y^{>N_m}\|x_m
 \le\|P_Y\|x_0+\sum_{m\ge1}2^{-m}x_m
 \lesssim_r\|b\|_{\ell^r}.
\]
Thus the measurable bounded formula supplies an allocation $d$ for
$B$ with uniformly bounded costs.  For the original column,
\[
 \|w_j/d_j\|_r\lesssim2^{-m}\qquad(j\in I_m),
\]
so the outer lift gives one analytic factorization with
$\|F_j\|_{H^p}\to0$.  Conversely, if $Y=AF$ and
$\|F_j\|_p\to0$, the first H\"older estimate gives
\[
 \|P_Y^{>N}\|
 \le\|A\|_{H^\infty(\ell^t)}^2
       \sup_{j>N}\|F_j\|_p^2\longrightarrow0.
\]
Then \eqref{eq:positive-tail-essential} makes $P_Y$ compact.  This
proves all the claims.
\end{proof}

\begin{definition}[Analytic Calder\'on symbol space]\label{def:calderon}
For an admissible resolution $\sigma$, put
\[
 \mathcal D_{p,\sigma}\phi
 =\bigl(2^{-j/p'}W_j^\sigma\phi\bigr)_{j\ge0}
 \tag{4.17}\label{eq:calderon-transform}
\]
and define
\[
\begin{split}
 \mathfrak C_p^\sigma
 &=\{\phi\in\Hol(\D):
     \mathfrak G_p(\mathcal D_{p,\sigma}\phi)<\infty\},\\
 \|\phi\|_{\mathfrak C_p^\sigma}
 &=|\phi(0)|+\mathfrak G_p(\mathcal D_{p,\sigma}\phi).
\end{split}
 \tag{4.18}\label{eq:calderon-norm}
\]
\end{definition}

\begin{theorem}[Intrinsic Calder\'on space and its little space]
\label{thm:calderon}
The set in Definition~\ref{def:calderon} is independent of the admissible
resolution, and all its norms are equivalent.  If
$\mathfrak M_p=(\mathcal R_p,H^p)$ denotes the Guo--Tang Hilbert-range
coefficient-multiplier space, then
\[
 \boxed{\mathfrak C_p^\sigma=\mathfrak M_p}
 \quad\text{with equivalent norms.}
 \tag{4.19}\label{eq:calderon-multiplier}
\]
Thus the common space may be denoted by $\mathfrak C_p$.

Its little space has the resolution-independent descriptions
\[
\boxed{
\begin{aligned}
 \mathfrak C_{p,0}
 &:=\overline{\operatorname{Poly}}^{\,\mathfrak C_p}\\
 &=\{\phi:\mathcal D_{p,\sigma}\phi
     \in H^\infty(\ell^t)\odot c_0(H^p)\}\\
 &=\{\phi\in\mathfrak C_p:
     \|\phi_\rho-\phi\|_{\mathfrak C_p}\to0
     \text{ as }\rho\uparrow1\},
\end{aligned}}
 \tag{4.20}\label{eq:calderon-little}
\]
where $\phi_\rho(z)=\phi(\rho z)$ and $\odot$ denotes coordinatewise
analytic product.  Finally, for $\phi=g'$,
\[
\begin{split}
 \mathcal H_g:H^p\to H^p\text{ is bounded}
 &\iff\phi\in\mathfrak C_p,\\
 \mathcal H_g\text{ is compact}
 &\iff\phi\in\mathfrak C_{p,0},
\end{split}
 \tag{4.21}\label{eq:calderon-hilbert}
\]
and
\[
 \|\mathcal H_g\|_{\mathrm e}
 \asymp_p
 \|\phi+\mathfrak C_{p,0}\|_{\mathfrak C_p/\mathfrak C_{p,0}}.
 \tag{4.22}\label{eq:calderon-quotient}
\]
\end{theorem}

\begin{proof}
Let $Y_j=2^{-j/p'}W_j^\sigma\phi$.  Boundary Khintchine gives, for
finite $a$,
\[
 \left\|\left(\sum_j|a_jY_j|^2\right)^{1/2}\right\|_p
 \lesssim_p
 \|\mathscr S_{\phi,p}^\sigma\|\,\|a\|_{\ell^p}.
\]
For the reverse inequality, colour the $j$'s modulo $2L+1$.  In one
colour their hard-scale supports are disjoint, and arbitrary signs on
those support intervals are uniformly bounded $H^p$ multipliers by
Lemma~\ref{lem:dyadic-marc}.  Khintchine in the reverse direction and
summation over the colours therefore give
\[
 \|\mathscr S_{\phi,p}^\sigma\|
 \asymp_{p,L}
 |\phi(0)|+\|P_Y:\ell^r\to L^r\|^{1/2}.
 \tag{4.23}\label{eq:smooth-square-column}
\]
The same proof, with the same constants, applies to every coordinate
tail.

Lemma~\ref{lem:allocation}, Theorem~\ref{thm:resolution}, and
Theorem~\ref{thm:dyadic} now give
\[
 \|T_\phi:H^p\to H^p\|
 \asymp_{p,L,M}\|\phi\|_{\mathfrak C_p^\sigma}.
\]
Guo--Tang identify the same operator norm with the multiplier norm on
$\mathfrak M_p$, proving \eqref{eq:calderon-multiplier}, completeness,
and resolution independence.

We now verify each description of the little space.  By the tail form of
\eqref{eq:smooth-square-column}, Theorem~\ref{thm:resolution}, and
Lemma~\ref{lem:allocation},
\[
 T_\phi\text{ is compact}
 \iff \|P_Y^{>N}\|\longrightarrow0
 \iff Y=AF\text{ for some }
 A\in H^\infty(\ell^t),\ F\in c_0(H^p).
 \tag{4.24}\label{eq:little-factor-compact}
\]
Let $S_N\phi$ retain the constant coefficient and the first $N+1$ hard
dyadic blocks.  It is a polynomial.  If $T_\phi$ is compact, the hard
tail criterion and the norm equivalences above give
\[
 \|\phi-S_N\phi\|_{\mathfrak C_p}\longrightarrow0.
 \tag{4.25}\label{eq:little-polynomial-approximation}
\]
Conversely, polynomial symbols induce finite-rank operators, while the
compact operators are norm closed.  The equivalence of the
$\mathfrak C_p$ norm and the $T_\phi$ operator norm therefore shows that
the compact-symbol class is precisely
$\overline{\operatorname{Poly}}^{\mathfrak C_p}$.  Together with
\eqref{eq:little-factor-compact}, this proves the first two lines of
\eqref{eq:calderon-little} and also their resolution independence.

For radial dilations, let $D_\rho f(z)=f(\rho z)$.  Then
$T_{\phi_\rho}=D_\rho T_\phi$.  For fixed $0<\rho<1$, $D_\rho$ is compact
on $H^p$: if $f(z)=\sum_na_nz^n$ and $\|f\|_p\le1$, then
$|a_n|\le1$, and hence
\[
 \left\|D_\rho f-\sum_{n=0}^N\rho^na_nz^n\right\|_p
 \le\sum_{n>N}\rho^n
 =\frac{\rho^{N+1}}{1-\rho}.
\]
Thus every $\phi_\rho$ belongs to $\mathfrak C_{p,0}$.  Moreover the
maps $\phi\mapsto\phi_\rho$ are uniformly bounded on $\mathfrak C_p$,
because $D_\rho$ is contractive on $H^p$ and the $\mathfrak C_p$ and
operator norms are equivalent.  They converge in norm on polynomials,
so by density they converge on $\mathfrak C_{p,0}$.  Conversely, if
$\|\phi_\rho-\phi\|_{\mathfrak C_p}\to0$, then $\phi$ is a norm limit of
the elements $\phi_\rho\in\mathfrak C_{p,0}$; closedness gives
$\phi\in\mathfrak C_{p,0}$.  This proves the radial line of
\eqref{eq:calderon-little}.

It remains to prove both directions of the quotient essential-norm
estimate.  If $\psi\in\mathfrak C_{p,0}$, then $T_\psi$ is compact, so
\[
 \|T_\phi\|_{\mathrm e}
 \le\|T_{\phi-\psi}\|
 \lesssim_p\|\phi-\psi\|_{\mathfrak C_p}.
\]
Taking the infimum over $\psi$ gives one inequality in
\eqref{eq:calderon-quotient}.  In the other direction,
$S_N\phi\in\mathfrak C_{p,0}$, and the hard/smooth tail comparisons and
Theorem~\ref{thm:main} give
\[
 \begin{split}
 \|\phi+\mathfrak C_{p,0}\|_{\mathfrak C_p/\mathfrak C_{p,0}}
 &\le\liminf_{N\to\infty}
      \|\phi-S_N\phi\|_{\mathfrak C_p}\\
 &\lesssim_p\lim_{N\to\infty}
      \|\mathscr S_{\phi,p}^{\sigma,>N}\|
 \lesssim_p\|T_\phi\|_{\mathrm e}.
 \end{split}
\]
This proves \eqref{eq:calderon-quotient} without assuming that a nearest
compact operator is itself a generalized Hilbert operator.
\end{proof}

\begin{remark}
The positive change-of-density principle behind
\eqref{eq:measurable-allocation} is classical, and analytic
regularization of Banach-function-space factorizations is developed in
\cite{LesnikMaligrandaMleczko2020}.  The contribution here is
its explicit analytic, countable-coordinate realization for the
generalized-Hilbert dyadic columns and its resolution-independent
bounded/compact symbol-space consequences.  No priority claim is made for
the abstract factorization mechanism.
\end{remark}

\section{Finite Poisson mixtures exactly norm the embedding}

\begin{lemma}[Poisson-mixture approximation]\label{lem:density}
If $W\in L^1_+$ and $\int W\,dm=1$, finite convex Poisson mixtures $Q_n$
can be chosen so that
\[
 \|Q_n-W\|_1\to0,\qquad
 \|Q_n^\alpha-W^\alpha\|_{1/\alpha}\to0.
 \tag{5.1}
\]
\end{lemma}

\begin{proof}
For $0<\rho<1$,
$P_\rho*W=\int_\T P_{\rho\eta}(\,\cdot\,)W(\eta)\,dm(\eta)$.
Continuity of $\eta\mapsto P_{\rho\eta}$ into $C(\T)$ allows uniform
approximation of this Bochner integral by finite convex combinations.
Let $\rho\uparrow1$ diagonally. The second convergence follows from
$|x^\alpha-y^\alpha|\le|x-y|^\alpha$.
\end{proof}

\begin{proof}[Proof of Theorem~\ref{thm:atomic}]
Every $h\in\mathcal A_K$ is an $H^t$ unit vector, giving one inequality.
Conversely, replace a normalized nonzero $h\in H^t$ by its outer factor
and put $W=|h|^t$. Let $|h_n|^t=Q_n$, with $Q_n$ from
Lemma~\ref{lem:density}. Then $|h_n|^2\to|h|^2$ in
$L^{1/\alpha}=L^{p/(p-2)}$. Since $|X_j|^2\in L^{p/2}$,
\[
 \|X_jh_n\|_2^2\to\|X_jh\|_2^2
\]
for each $j$. Apply this first to a finite coordinate set and then enlarge
the set. This proves \eqref{eq:atomic-exact}, including the infinite case. Restricting
all coordinates to $j>N$ proves the tail identity.
\end{proof}

\section{No fixed finite test level suffices}

\begin{lemma}[Poisson kernels away from their angular cell]
\label{lem:poisson-cell}
Partition $\T$ into $N\ge2$ half-open arcs $I_0,\ldots,I_{N-1}$ of
equal normalized length $1/N$.  For each $a\in\D\setminus\{0\}$ mark
the cell containing $a/|a|$ and its two circular neighbors.  If $I$ is
unmarked for each of $a_1,\ldots,a_K$, then, for nonnegative $c_\nu$,
\[
 \sup_I\sum_{\nu=1}^Kc_\nu P_{a_\nu}
 \le C\inf_I\sum_{\nu=1}^Kc_\nu P_{a_\nu},
 \tag{6.1}\label{eq:poisson-cell-harnack}
\]
where $C$ is absolute.  Kernels with $a_\nu=0$ may be included without
marking any cell. Consequently, if $0<\beta\le1$ and
$Q=\sum_\nu c_\nu P_{a_\nu}$, then
\[
 \sup_I Q^\beta\le C N\int_IQ^\beta\,dm.
 \tag{6.2}\label{eq:poisson-cell-average}
\]
\end{lemma}

\begin{proof}
Write $a=\rho e^{i\psi}$ and use circular angular distance
$d(\theta,\psi)\in[0,\pi]$.  If $I$ is not one of the three marked
cells, then for $e^{i\theta},e^{i\vartheta}\in I$,
\[
 d(\theta,\psi)\ge\frac{c}{N},
 \qquad
 |d(\theta,\psi)-d(\vartheta,\psi)|\le\frac{2\pi}{N}.
 \tag{6.3}\label{eq:cell-distance}
\]
When the smaller distance is at most $\pi/4$, the two distances are
comparable and both are at most $3\pi/4$, with an absolute comparison
constant.  When it is larger than $\pi/4$, both corresponding sines are
bounded below by an absolute constant. In either case,
\[
 (1-\rho)^2+4\rho\sin^2\frac{\theta-\psi}{2}
 \asymp
 (1-\rho)^2+4\rho\sin^2\frac{\vartheta-\psi}{2}.
\]
Since the two sides are the squared denominators of $P_a$, this proves
uniform comparability for each kernel. Summing with nonnegative
coefficients proves \eqref{eq:poisson-cell-harnack}; $P_0\equiv1$ causes
no change. Raising to the power $\beta$ and using $m(I)=1/N$ proves
\eqref{eq:poisson-cell-average}.
\end{proof}

\begin{lemma}[Dirichlet--Poisson micro-arcs]
\label{lem:micro-arcs}
Let $M,N\ge2$, put $\rho=1-M^{-1}$ and $s=\rho^{1/N}$.  There is an
absolute $c_0>0$ such that, for every $N$th root of unity $\omega$, the
arc
\[
 J_\omega=\{e^{i\theta}:|\theta-\arg\omega|_{\T}\le c_0/(MN)\}
\]
satisfies
\[
 |D_M(e^{iN\theta})|\ge cM,
 \qquad P_{s\omega}(e^{i\theta})\ge cMN.
 \tag{6.4}\label{eq:microarc-lower}
\]
The arcs $J_\omega$ are pairwise disjoint.
\end{lemma}

\begin{proof}
If $\delta=\theta-\arg\omega$ is chosen in $[-\pi,\pi]$, then
$|N\delta|\le c_0/M$.  Taking $c_0$ sufficiently small and using
\[
 |D_M(e^{ix})|=\left|\frac{\sin(Mx/2)}{\sin(x/2)}\right|
\]
gives the first inequality in \eqref{eq:microarc-lower}.  Moreover,
the mean value theorem applied to $x^{1/N}$ on $[1/2,1]$ gives
\[
 1-s=1-(1-M^{-1})^{1/N}\asymp(MN)^{-1}.
 \tag{6.5}\label{eq:s-radius}
\]
Thus, on $J_\omega$,
\[
 |1-se^{i\delta}|^2
 =(1-s)^2+4s\sin^2(\delta/2)\lesssim(MN)^{-2},
 \qquad 1-s^2\asymp(MN)^{-1},
\]
which proves the Poisson lower bound. Finally, distinct roots have
angular separation $2\pi/N$, whereas each $J_\omega$ has length
$2c_0/(MN)<2\pi/N$.
\end{proof}

\begin{proposition}[Sharp composite Dirichlet capture]
\label{prop:capture}
For $M,N\ge2$, $K\ge1$, and
$X(z)=dM^{-1/2}z^LD_M(z^N)$,
\[
 \sup_{h\in\mathcal A_K}\|Xh\|_2
 \asymp_p d\left[1+M^{1/t}\min\{1,(K/N)^{1/p}\}\right].
 \tag{6.6}
\]
Also $\|M_X:H^t\to H^2\|=\|X\|_p\asymp_pdM^{1/t}$.
\end{proposition}

\begin{proof}
The multiplier identity follows from H\"older and the outer unit vector
with $|h|^t=|X|^p/\|X\|_p^p$; the standard estimate
$\|D_M\|_p\asymp_pM^{1-1/p}$ gives its value.

Put $F(\zeta)=|D_M(\zeta^N)|^2$ and partition $\T$ into its $N$ period
cells $I_\ell$. Since $2/(1-\alpha)=p$,
\[
 \left(\int_{I_\ell}F^{1/(1-\alpha)}\,dm\right)^{1-\alpha}
 \asymp_pM^{1+\alpha}N^{-(1-\alpha)}.
 \tag{6.7}
\]
For $Q=\sum_{\nu=1}^Kc_\nu P_{a_\nu}$, mark the cell containing the
angular direction of each $a_\nu$ and its two neighbors. At most
$\min(N,3K)$ cells are marked. H\"older and concavity, with
$q_\ell=\int_{I_\ell}Q\,dm$, give on their union $E$
\[
 \int_EFQ^\alpha dm\lesssim_p
 M^{1+\alpha}\left(\frac{\min(K,N)}N\right)^{1-\alpha}.
 \tag{6.8}\label{eq:capture-marked}
\]
On an unmarked cell, Lemma~\ref{lem:poisson-cell} gives
$\sup_{I_\ell}Q^\alpha\lesssim N\int_{I_\ell}Q^\alpha dm$.
Since $\int_{I_\ell}F\,dm=M/N$ and $\int Q^\alpha\,dm\le1$,
\[
 \int_{\T\setminus E}FQ^\alpha dm\lesssim_pM.
 \tag{6.9}\label{eq:capture-unmarked}
\]
Equations \eqref{eq:capture-marked}--\eqref{eq:capture-unmarked},
multiplied by $d^2/M$, prove the
upper estimate after taking square roots.

The constant test gives $\|X\|_2=d$. If $K\le N$, put
$\rho=1-1/M$, $s_0=\rho^{1/N}$, select $K$ distinct $N$th roots
$\omega_\nu$, and set $Q_K=K^{-1}\sum_{\nu=1}^KP_{s_0\omega_\nu}$.
Lemma~\ref{lem:micro-arcs} supplies disjoint arcs of normalized length
comparable to $(MN)^{-1}$ around those roots on which
$F\gtrsim M^2$ and $Q_K\gtrsim MN/K$. Thus
\[
 \int FQ_K^\alpha dm\gtrsim_p
 M^{1+\alpha}(K/N)^{1-\alpha}.
 \tag{6.10}\label{eq:capture-lower}
\]
This gives the second lower term because $\alpha/2=1/t$ and
$(1-\alpha)/2=1/p$. For $K\ge N$, use all $N$ roots.
\end{proof}

\begin{theorem}[Fixed-level and kernel-test failure]\label{thm:failure}
For every fixed $K_0\ge1$ there is $\phi\in\Hol(\D)$ such that
\[
 \sup_{h\in\mathcal A_{K_0}}
 \left(\sum_j\|X_jh\|_2^t\right)^{1/t}<\infty,
 \tag{6.11}\label{eq:fixed-level-test}
\]
with a uniformly vanishing coordinate tail, while $T_\phi$ is unbounded.
Also $(\|X_j\|_2)_j\in\ell^t$ with a vanishing tail. For $K_0=1$,
the test family is precisely the normalized kernels
\[
 \kappa_{a,t}(z)=
 \frac{(1-|a|^2)^{1/t}}{(1-\overline az)^{2/t}},
 \qquad |\kappa_{a,t}|^t=P_a.
 \tag{6.12}
\]
\end{theorem}

\begin{proof}
Choose $d_m=2^{-m}$ and $M_m\uparrow\infty$ so fast that
$d_mM_m^{1/t}\to\infty$. Choose
$N_m\ge K_0M_m^{p/t}$. At separated hard-block positions $J_m$, set
\[
 X_{J_m}(z)=\frac{d_m}{\sqrt{M_m}}z^{2^{J_m}}
 D_{M_m}(z^{N_m}),
 \tag{6.13}
\]
with $(M_m-1)N_m<2^{J_m}$, and put $X_j=0$ otherwise. Increase $J_m$
further so that $\phi=\sum_m2^{J_m/p'}X_{J_m}$ converges locally
uniformly; exponential radial decay on compact subdisks permits this.

Proposition~\ref{prop:capture} gives
$\sup_{h\in\mathcal A_{K_0}}\|X_{J_m}h\|_2\lesssim_pd_m$.
The $\ell^t$ sum proves \eqref{eq:fixed-level-test} and its vanishing tail. Also
$\|X_{J_m}\|_2=d_m$. In contrast,
$\|M_{X_{J_m}}:H^t\to H^2\|\asymp_pd_mM_m^{1/t}\to\infty$.
If $\Ccal_{\phi,p}$ were bounded, all its coordinate multiplier norms
would be uniformly bounded. Thus it, and hence $T_\phi$, is unbounded.
\end{proof}

\section{Intrinsic aggregate complexity}

For any sequence $X=(X_j)$ of analytic polynomials define
\[
 \Ccal_K(X)=\sup_{h\in\mathcal A_K}
 \left(\sum_j\|X_jh\|_2^t\right)^{1/t},
 \quad \Ccal_\infty(X)=\|h\mapsto(X_jh)_j\|.
 \tag{7.1}\label{eq:atomic-norms}
\]
Let $\mathcal P_K$ be the probability densities formed from at most $K$
Poisson kernels. Let $\Lambda$ consist of finite nonnegative
$\lambda$ with $\|\lambda\|_{\ell^r}=1$. For
\eqref{eq:aggregate-def-intro}, omitting
$A_\lambda=0$, put
\[
 \delta_K(\lambda)=\inf_{Q\in\mathcal P_K}\|W_\lambda-Q\|_1,
 \quad
 \gamma_K(\lambda)=\sup_{Q\in\mathcal P_K}
 \int W_\lambda^{1-\alpha}Q^\alpha dm.
 \tag{7.2}
\]

\begin{theorem}[Exact aggregate-density formula]\label{thm:aggregate}
For every $K\ge1$,
\[
 \boxed{\Ccal_K(X)^2=\sup_{\lambda\in\Lambda}
 A_\lambda\gamma_K(\lambda),\qquad
 \Ccal_\infty(X)^2=\sup_{\lambda\in\Lambda}A_\lambda.}
 \tag{7.3}\label{eq:aggregate-exact}
\]
There is $c_p>0$ such that
\[
 c_p\delta_K(\lambda)^2\le1-\gamma_K(\lambda)
 \le\delta_K(\lambda)^\alpha.
 \tag{7.4}\label{eq:affinity-distance}
\]
Consequently $\Ccal_K(X)\uparrow\Ccal_\infty(X)$, including an infinite
limit. All statements remain valid for coordinate tails.
\end{theorem}

\begin{proof}
For $y_j=\|X_jh\|_2^2$, the identity $t/2=r'$ gives
\[
 \left(\sum_j\|X_jh\|_2^t\right)^{2/t}
 =\sup_{\lambda\in\Lambda}\sum_j\lambda_j\|X_jh\|_2^2.
 \tag{7.5}
\]
If $|h|^t=Q$, the right side for fixed $\lambda$ is
$\int S_\lambda Q^\alpha dm$. The suprema commute and
$S_\lambda/A_\lambda=W_\lambda^{1-\alpha}$, proving the first formula.
H\"older gives an upper bound $A_\lambda$, with formal optimizer
$Q=W_\lambda$. Lemma~\ref{lem:density} approximates that optimizer and
proves the second formula.

For probability densities $W,Q$,
$|x^\alpha-y^\alpha|\le|x-y|^\alpha$ and H\"older yield
\[
 1-\int W^{1-\alpha}Q^\alpha dm\le\|W-Q\|_1^\alpha.
 \tag{7.6}
\]
On the other hand,
\[
 (1-\alpha)u^2+\alpha v^2-u^{2(1-\alpha)}v^{2\alpha}
 \ge c_\alpha(u-v)^2.
 \tag{7.7}\label{eq:uniform-convexity}
\]
After homogeneity reduces to $u=1$, the quotient by $(1-v)^2$ extends
positively at $0,1,\infty$. Apply \eqref{eq:uniform-convexity} to
$u=\sqrt W$, $v=\sqrt Q$, and use
$\|W-Q\|_1\le2\|\sqrt W-\sqrt Q\|_2$ to obtain the reverse bound in
\eqref{eq:affinity-distance}.

The classes $\mathcal P_K$ are nested. Given
$B<\Ccal_\infty(X)^2$, choose finite $\lambda$ with $A_\lambda>B$.
Lemma~\ref{lem:density} gives $\delta_K(\lambda)\to0$, so
\eqref{eq:aggregate-exact}--\eqref{eq:affinity-distance} imply
$\Ccal_K(X)^2>B$ eventually. Finite
support of $\lambda$ also justifies the identical tail argument.
\end{proof}

For one nonzero polynomial $X$, set
$W_X=|X|^p/\|X\|_p^p$. If $\delta_K(X)$ is its $L^1$ distance to
$\mathcal P_K$ and
$\Gamma_K(X)=(\sup_{h\in\mathcal A_K}\|Xh\|_2/\|X\|_p)^2$, then
\[
 c_p\delta_K(X)^2\le1-\Gamma_K(X)
 \le\delta_K(X)^\alpha.
 \tag{7.8}
\]
Thus single-block atomic complexity is exactly, up to the displayed
accuracy conversion, Poisson-mixture approximation complexity of its
$H^p$ mass density.

\section{A finite-bandwidth atomic budget}

\begin{lemma}[BV Poisson discretization]\label{lem:bv}
If $W$ is a probability density on $\T$ of total variation $V$, then
\[
 \inf_{Q\in\mathcal P_K}\|W-Q\|_1
 \lesssim\min\left\{1,
 \sqrt{\frac{V+1}{K}}\log(eK(V+1))\right\}.
 \tag{8.1}\label{eq:bv-budget}
\]
\end{lemma}

\begin{proof}
We give the details, with angles represented in $[-\pi,\pi]$.  Write
$\tau_\theta W(e^{ix})=W(e^{i(x-\theta)})$. The translation estimate for
a function of bounded variation is
\[
 \|\tau_\theta W-W\|_1\le C V|\theta|.
 \tag{8.2}\label{eq:bv-translation}
\]
For smooth $W$ this follows from the fundamental theorem of calculus and
$\|W'\|_1\le CV$; the general case follows by convolution with a smooth
approximate identity, whose variations do not exceed $V$, followed by
$L^1$ convergence. For general Hardy-space and Poisson-kernel background
see \cite{Garnett2007,JevticVukoticArsenovic2016}; the quantitative
estimates needed here are proved explicitly below.

Put $\rho=1-\varepsilon$, where $0<\varepsilon\le1/2$.  Minkowski's
inequality and \eqref{eq:bv-translation} give
\[
 \|P_\rho*W-W\|_1
 \le\int_{-\pi}^{\pi}P_\rho(e^{i\theta})
       \|\tau_\theta W-W\|_1\,dm(\theta)
 \le CV\int_{-\pi}^{\pi}|\theta|P_\rho(e^{i\theta})\,dm(\theta).
 \tag{8.3}\label{eq:bv-poisson-minkowski}
\]
The identity
\[
 P_\rho(e^{i\theta})
 =\frac{1-\rho^2}{(1-\rho)^2+4\rho\sin^2(\theta/2)}
 \le\frac{C\varepsilon}{\varepsilon^2+\theta^2}
 \tag{8.4}\label{eq:poisson-upper}
\]
shows, by direct integration over $[0,\pi]$, that
\[
 \|P_\rho*W-W\|_1
 \le CV\varepsilon\log(e/\varepsilon).
 \tag{8.5}\label{eq:bv-smoothing}
\]

Now partition $\T$ into $K$ equal arcs $E_\nu$. Choose
$\eta_\nu\in E_\nu$, put $c_\nu=\int_{E_\nu}W\,dm$, and set
$Q_{K,\rho}=\sum_\nu c_\nu P_{\rho\eta_\nu}$.  The angular Poisson
kernel is even and decreases on $[0,\pi]$, and therefore
\[
 \|P_\rho'\|_1
 =\frac{P_\rho(1)-P_\rho(-1)}{\pi}\le\frac{C}{\varepsilon}.
 \tag{8.6}\label{eq:poisson-derivative}
\]
Translation and the fundamental theorem of calculus consequently imply
\[
\begin{split}
 \|P_\rho*W-Q_{K,\rho}\|_1
 &\le\sum_{\nu=1}^K\int_{E_\nu}W(\eta)
  \|P_{\rho\eta}-P_{\rho\eta_\nu}\|_1\,dm(\eta)\\
 &\le \frac{C}{K}\|P_\rho'\|_1
 \le\frac{C}{K\varepsilon}.
\end{split}
 \tag{8.7}\label{eq:poisson-discretization}
\]
Take $\varepsilon=[K(V+1)]^{-1/2}$ when this is at most $1/2$.
Combining \eqref{eq:bv-smoothing} and
\eqref{eq:poisson-discretization} yields the second term in
\eqref{eq:bv-budget}.  When this choice exceeds $1/2$, use
$\|W-Q\|_1\le2$ for probability densities and enlarge the absolute
constant.
\end{proof}

\begin{theorem}[Finite-bandwidth approximation]\label{thm:bandwidth}
Let $X_1,\ldots,X_m$ be a finite nonzero analytic polynomial family and
assume each $|X_j|^2$ has trigonometric degree at most $D$. If
$\Ccal_K,\Ccal_\infty$ are defined by \eqref{eq:atomic-norms} and
$\Ccal_\infty>0$, then \eqref{eq:bandwidth-intro} holds. Hence, up to the
displayed logarithmic factor, $K\gtrsim_p(D+1)\eta^{-2p/(p-2)}$ captures the
squared norm with relative error at most $\eta$.
\end{theorem}

\begin{proof}
The finite-dimensional supremum
$\Ccal_\infty^2=\sup_{\lambda\in\Lambda}A_\lambda$ is attained. For a
maximizer, $S_\lambda$ is a nonnegative trigonometric polynomial of degree
at most $D$. Bernstein's inequality and H\"older give
\[
 \|W_\lambda'\|_1
 \le r\frac{\|S_\lambda\|_r^{r-1}\|S_\lambda'\|_r}
 {\|S_\lambda\|_r^r}\le rD.
\]
Lemma~\ref{lem:bv} bounds $\delta_K(\lambda)$, and
Theorem~\ref{thm:aggregate} gives
$\Ccal_K^2\ge\Ccal_\infty^2[1-\delta_K(\lambda)^\alpha]$.
\end{proof}

\section{Consequences and scope}

Theorem~\ref{thm:main} is a necessary-and-sufficient condition for every
analytic symbol at every fixed $2<p<\infty$, including essential norms
and compactness. Taking $h=1$ recovers the $H^p\to H^2$ condition. For a
one-frequency-per-block symbol the constant test is norming, recovering
the lacunary result. H\"older also shows that
$(\|X_j\|_p)_j\in\ell^t$ is sufficient, recovering the known lower
inclusion.

Theorem~\ref{thm:resolution} removes dependence on the hard cutoff:
boundedness, tail norms, compactness, and essential norms are unchanged,
up to structural constants, under every admissible smooth analytic
resolution. Theorem~\ref{thm:calderon} then identifies the resulting
resolution-independent analytic Calder\'on space with the Guo--Tang
multiplier space and identifies its little space by polynomial closure,
analytic $c_0(H^p)$ factorization, and radial norm continuity.

Theorem~\ref{thm:dense-separation} shows that this sufficient inclusion is
strict in a robust way: the new bounded symbol has consecutive active
blocks, occupies half of every large hard block, and is not supported on
an analytic $\Lambda(p)$-set.

Theorems~\ref{thm:atomic} and \ref{thm:failure} show, respectively, that
all finite positive Poisson mixtures are complete and that no fixed
finite level is complete. Proposition~\ref{prop:capture} locates the sharp
complexity scale for a natural multi-peak model. Theorem
\ref{thm:aggregate} replaces model-specific peak counts by intrinsic
near-norming aggregate densities, and Theorem~\ref{thm:bandwidth} makes
that reduction quantitative for finite angular bandwidth.

The criterion and $\mathfrak C_p$ are explicit analytic vector/product
descriptions, not an identification with a previously named scalar Besov,
Lipschitz, tent, or mixed-norm space. Eliminating the full family of outer
tests by a comparably classical scalar norm remains a separate
scalarization problem.
All arguments are analytic. Targeted searches did not find the exact
package in the direct literature; MathSciNet, zbMATH, the formula-level
chain of the coefficient-multiplier literature, unpublished manuscripts,
and later revisions remain to be checked. No global first-proof claim is
made.

\end{document}